\documentclass[11pt]{amsart}
\usepackage[margin=1.3in]{geometry}
\usepackage[utf8]{inputenc}
\usepackage{amsmath}
\usepackage[pagebackref=true]{hyperref}
\usepackage{amstext}
\usepackage{amscd}
\usepackage{amsmath}
\usepackage{amsthm}
\usepackage{latexsym}
\usepackage{url}
\usepackage{amssymb}
\usepackage[all]{xy}
\usepackage{mathrsfs} 
\usepackage{color}
\usepackage{tikz}
\usepackage{indentfirst}
\usepackage{MnSymbol}
\usepackage{combelow}

\newtheorem{theorem}{Theorem}[section]
\newtheorem{thmx}{Theorem}
 
\newtheorem{lemma}[theorem]{Lemma}

\newtheorem{proposition}[theorem]{Proposition}
\newtheorem*{theorem*}{Theorem}

\theoremstyle{definition}
\newtheorem{definition}[theorem]{Definition}
\newtheorem{theorem-definition}[theorem]{Theorem-Definition}

\theoremstyle{remark}
\newtheorem{remark}[theorem]{Remark}
\newtheorem{example}[theorem]{Example}
\newtheorem{question}[theorem]{Question}

\numberwithin{equation}{section}

\hypersetup{
     colorlinks,breaklinks,
     urlcolor = [RGB]{0,50,175},
     citecolor = [RGB]{0,50,175},
     linkcolor =[RGB]{150,10,0}
}

\begin{document}
\title{Rational curves on hypersurfaces}
\author{Yuan Wang}
\address{Department of Mathematics, University of Utah, 155 South 1400 East, Salt Lake City, UT 84112-0090, USA}
\email{ywang@math.utah.edu}
\thanks{The author was supported in part by the FRG grant DMS-\#1265261.}
\begin{abstract}
Let $(X,D)$ be a pair where $X$ is a projective variety. We study in detail how the behavior of rational curves on $X$ as well as the positivity of $-(K_X+D)$ and $D$ influence the behavior of rational curves on $D$. In particular we give criteria for uniruledness and rational connectedness of components of $D$.
\end{abstract}
\maketitle
\section{introduction}
For a projective variety $X$, the connection between positivity of $-K_X$ and the behavior of rational curves on $X$ is well understood. Uniruledness and rational connectedness are possibly the two birational properties of smooth varieties that have been the most intensively studied. A result of Miyaoka-Mori \cite{MM86} shows that a smooth projective variety $X$ is uniruled if and only if there exists a $K_X$-negative curve through every general point of $X$. Later Boucksom-Demailly-P\u{a}un-Peternell \cite{BDPP13} proved that if the canonical divisor of a projective manifold $X$ is not pseudo-effective, then $X$ is uniruled. The rational connectedness of smooth Fano varieties was established by Campana \cite{Campana92} and Koll\'{a}r-Miyaoka-Mori \cite{KMM92}, and it was later generalized to the log Fano cases by Zhang \cite{Zhang06} and Hacon-M\textsuperscript{c}Kernan \cite{HM07}.

A natural question is how the behavior or rational curves on a variety $X$ influences the behavior of rational curves on a hypersurface $D$. An easy case is when $X=\mathbb{P}^n$, then a general hypersurface of degree $\le n$ is rationally connected. More generally if $(X,D)$ is a plt pair and $-(K_X+D)$ is ample, then by the adjunction formula we have $(K_X+D)|_D=K_D+{\rm Diff}_D(0)$, which is anti-ample and klt. So by \cite[Theorem 1]{Zhang06} $D$ is rationally connected, in particular uniruled. However if we assume that $-(K_X+D)$ is big and semiample instead of ample, then the following example shows that $D$ is not necessarily uniruled. 
\begin{example}\label{ruledsurface}
Let $\pi: X=\mathbb{P}(\mathcal{E})\to C$ be a ruled surface where $C$ is an elliptic curve and $\mathcal{E}=\mathcal{O}_C\oplus\mathcal{L}$ such that $\mathcal{L}$ is a line bundle on $C$ and ${\rm deg}(\mathcal{L})<0$. Let $e=-{\rm deg}(\bigwedge^2\mathcal{E})$, then $e>0$ and $K_X\equiv_{\rm num} -2C_0-eF$ where $C_0$ is the unique section of $\pi$ with $\mathcal{O}_X(C_0)\cong \mathcal{O}_X(1)$ (see \cite[Ch. V, Example 2.11.3]{Hartshorne77}) and $F$ is a fiber. So we have
$$-(K_X+C_0)\equiv_{\rm num}C_0+eF=\epsilon C_0+(1-\epsilon)(C_0+\dfrac{e}{1-\epsilon}F)$$ 
where $\epsilon\in(0,1)$ is any rational number. Now $C_0+\dfrac{e}{1-\epsilon}F$ is ample by \cite[Ch. V, Proposition 2.20]{Hartshorne77} and $C_0$ is effective, so $-(K_X+C_0)$ is nef and big. Moreover by \cite[Theorem 1.7]{Gongyo12} we know that $-(K_X+C_0)$ is semiample. However $C_0$ is an elliptic curve, in particular not uniruled.
\end{example}
In this paper, first we give a criterion for uniruledness of $D$. Roughly speaking we show that if $X$ contains ``sufficiently many" rational curves, then as long as $K_X+D$ is not pseudo-effective the uniruledness of $D$ holds. More precisely we have
\begin{thmx}[Theorem \ref{hypersurfaceuniruledanydimension}]\label{1}
Let $(X,D)$ be a pair where $D=\sum_iE_i+\sum_ja_jF_j$ such that $E_i$ and $F_j$ are distinct prime divisors and $a_j\in (0,1)$. Suppose that ${\rm rd}(X)\ge 2$ and $K_X+D$ is not pseudo-effective, then $E_i$ is uniruled for any $i$.
\end{thmx}
Here ${\rm rd}(X)$ is the rational dimension of $X$, which is the dimension of the general fiber of the maximal rationally connected fibration of $X$ (see Definition \ref{rd}). 

The author suspects that Theorem \ref{1} is already sharp. First note that we do not have any assumption on the singularities of the pair $(X, D)$ in Theorem \ref{1}. Next, Example \ref{ruledsurface} shows that the condition ${\rm rd}(X)\ge 2$ cannot be weakened even when $K_X+D$ is very negative (e.g. anti-big and anti-semiample). Finally, the following simple example indicates that the condition that $K_X+D$ is not pseudo-effective cannot be weakened either. 
\begin{example}
Let $C\subset \mathbb{P}^2$ be an elliptic curve of degree $3$. Then we have $K_{\mathbb{P}^2}+C\sim_{\rm lin}0$. Let $f:X\to \mathbb{P}^2$ be the blow-up of $X$ at a point not in $C$. We have
$$K_X+f_*^{-1}C=f^*(K_{\mathbb{P}^2}+C)+E,$$
where $E$ is the exceptional divisor. Now $X$ is a rational surface, in particular ${\rm rd}(X)=2$. $K_X+f_*^{-1}C$ is pseudo-effective and yet not nef. In this case $f_*^{-1}C$ is not a rational curve, hence not uniruled.
\end{example}

The strategy to prove Theorem \ref{1} is to use the minimal model program in arbitrary dimension developed in \cite{BCHM10} as well as an induction on the dimension of $X$. 

Note that \cite[Theorem 3.7]{LZ15} implies Theorem \ref{1} in the case where $(X,D)$ is dlt. This was pointed out by De-Qi Zhang after the completion of this paper.

Motivated by Theorem \ref{1} we also consider rational connectedness of hypersurfaces and obtain the following
\begin{thmx}[Theorem \ref{Dbig}, \ref{P3} and \ref{P4}] 
Let $(X,D)$ be a pair where $D = E+ \sum_j a_jF_j$ such that $\lfloor D\rfloor = E$ and $F_j$ are discinct prime divisors and $a_i \in [0, 1)$. Assume that $(X,D)$ is plt. Suppose that we are in one of the following cases.
\begin{enumerate}
\item $X$ is rationally connected, $\lfloor D\rfloor $ is big and $K_X+D$ is not pseudo-effective.
\item $X$ is a rationally connected threefold, $D$ is a prime divisor and $-(K_X+D)$ is Cartier, nef and big.
\item $X$ is a toric variety and $-(K_X+D)$ is big and semiample.
\end{enumerate}
Then $\lfloor D\rfloor$ is rationally connected.
\end{thmx}
\subsection*{Acknowledgements} 
The author offers his special thanks to Honglu Fan who suggests this topic along with many interesting ideas, and Christopher Hacon who provides lots of valuable advices. He would also like to thank Mircea Musta\cb{t}\u{a} for helpful discussions. Moreover he thanks De-Qi Zhang for pointing out the close relationship between Theorem \ref{1} and \cite[Theorem 3.7]{LZ15}. Finally he is grateful to the referees for many useful suggestions and comments.
\section{preliminaries}
In this paper we work over the field of the complex number $\mathbb{C}$. We will freely use the standard notations in \cite[especially 3.G]{HK10} (e.g. pair, discrepancy and klt, plt, dlt, lc singularities). Terms such as uniruled, rationally connected (RC) and rationally chain connected (RCC) will also be used and their definitions can be found in \cite{Kollar96}. The following definition can be found in \cite{Harris} by Harris.
\begin{definition}\label{rd}
Let $X$ be a proper smooth variety and $f:X\dasharrow Z$ the maximal rationally connected fibration (cf. \cite[Definition 5.3]{Kollar96}). We define the \emph{rational dimension} of $X$ to be ${\rm rd}(X):={\rm dim}(X)-{\rm dim(Z)}$. If $X$ is singular we define the rational dimension of $X$ to be ${\rm rd}(\tilde{X})$ for some resolution $\mu:\tilde{X}\to X$ of $X$.
\end{definition}
Next we present two theorems that are essential in the proof of Theorem \ref{1}. We have the following definition of minimal dlt model.
\begin{definition} \cite[Definitions and Notation 1.9]{KK10}
Let $(X,D)$ be a pair and $f^m : X^m\to X$ a proper birational morphism such that 
$$K_{X^m}+(f^m)_*^{-1}D=(f^m)^*(K_X+D)+\sum_ia_iE_i.$$
Let $D^m:=(f^m)^{-1}D+\sum_{a_i\le -1}E_i$. Then $(X^m,D^m)$ is a \emph{minimal dlt model} of $(X,D)$ if it is a dlt pair and the discrepancy of every $f^m$-exceptional divisor is at most $-1$. 
\end{definition}
\begin{theorem}[Dlt modification, by Hacon]\cite[Theorem 3.1]{KK10}\label{dltmod}
Let $(X,D)$ be a pair such that $X$ is quasi-projective, $D$ is a boundary, and $K_X+D$ is a $\mathbb{Q}$-Cartier divisor. Then $(X, D)$ admits a $\mathbb{Q}$-factorial minimal dlt model: $(X^m,D^m)\to(X,D)$. In particular if $K_X+D$ is not pseudo-effective then $K_{X^m}+D^m$ is also not pseudo-effective.
\end{theorem}
\begin{remark} \label{dltpreserve}
The reason for the second statement in Theorem \ref{dltmod} is the following. We have that $f^m$ only extracts divisors with discrepancy $\le -1$. So by definition of $D^m$ we can write
$$K_{X^m}+D^m=f^*(K_X+D)+\sum_jb_jE_j$$
where $b_j\le 0$. Therefore the second statement holds. 
\end{remark}
The second theorem is the existence of Mori fiber space established by Birkar-Cascini-Hacon-McKernan. For convenience we give the definition of Mori fiber space here.
\begin{definition} \cite[Definition 3.10.7]{BCHM10} \label{DefMFS}
Let $(X, \Delta)$ be a log canonical pair and $f : X \to Z$ be a projective morphism of normal varieties. Then $f$ is a \emph{Mori fibre space} if
\begin{enumerate}
\item $X$ is $\mathbb{Q}$-factorial and $\Delta$ is an $\mathbb{R}$-divisor,
\item $f$ is a contraction morphism, $\rho(X/Z) = 1$ and $\dim Z < \dim X$, and 
\item $-(K_X + \Delta)$ is $f$-ample.
\end{enumerate}
\end{definition}
\begin{theorem}[Existence of Mori fiber space]\cite[Corollary 1.3.3]{BCHM10}\label{ExistMFS}
Let $(X,\Delta)$ be a $\mathbb{Q}$-factorial klt pair. Let $\pi:X\to U$ be a projective morphism of normal quasi-projective varieties. Suppose that $K_X+\Delta$ is not $\pi$-pseudo-effective. Then we can run $(K_X+\Delta)$-minimal model program over $U$ which ends with a Mori fiber space over $U$.
\end{theorem}
Finally in this section we provide the following lemma which is known to experts.
\begin{lemma}\label{fiberklt}
Let $(X,D)$ be a klt pair. Suppose that we have a morphism $f:X\to Y$ such that ${\rm dim}(Y)<{\rm dim}(X)$ and $f_*\mathcal{O}_X=\mathcal{O}_Y$. Then for a general fiber $F$ of $f$, $(F,D|_F)$ is klt.
\end{lemma}
\begin{proof}
We take a log resolution for $(X,D)$ which we denote by $\mu:X'\to X$, and define $D'$ as 
$$K_{X'}+D'=\mu^*(K_X+D).$$
We write $D'=\Gamma'-E'$ where $\Gamma'$ and $E'$ are effective $\mathbb{Q}$-divisors which do not have common components. Let $f'=f\circ \mu$ and $F'$ a general fiber of $f'$ which maps to a general fiber of $f$ through $\mu$, then we have the following diagram
\begin{center}
\begin{tikzpicture}[scale=1.6]
\node (A) at (0,0) {$X'$};
\node (B) at (1,0) {$X$};
\node (C) at (0,1) {$F'$};
\node (D) at (1,1) {$F$};  
\path[->,font=\scriptsize]
(A) edge node[above]{$\mu$} (B)
(C) edge node[above]{$\nu$} (D);
\draw[->,font=\scriptsize]
(C) edge node[right]{} (A)
(D) edge node[right]{} (B);
\end{tikzpicture} 
\end{center}    
Since $\Gamma'|_{F'}$ is simple normal crossing we have that
$$(K_{X'}+\Gamma')|_{F'}=K_{F'}+\Gamma'|_{F'}$$
is klt. So
$$\nu^*(K_F+D|_F)=\mu^*(K_X+D)|_{F'}=(K_{X'}+D')|_{F'}=K_{F'}+\Gamma'-E'$$
is sub-klt. Therefore $K_F+D|_F$ is klt.
\end{proof}
\section{Uniruledness of hypersurfaces}
The main theorem of this section is as follows.
\begin{theorem}\label{hypersurfaceuniruledanydimension}
Let $(X,D)$ be a pair where $D=\sum_iE_i+\sum_ja_jF_j$ such that $E_i$ and $F_j$ are distinct prime divisors and $a_j\in (0,1)$. Suppose that ${\rm rd}(X)\ge 2$ and $K_X+D$ is not pseudo-effective, then $E_i$ is uniruled for any $i$.
\end{theorem}
By Theorem \ref{dltmod} and Remark \ref{dltpreserve} we can assume that $(X,D)$ dlt and $\mathbb{Q}$-factorial by possibly doing a dlt modification. We first consider the case when ${\rm dim}(X)=2$. Note that in this case ${\rm rd}(X)\ge 2$ is equivalent to that $X$ is RC.
\begin{lemma}\label{RCCsurface}
Let $(X,D)$ be dlt pair where ${\rm dim}(X)=2$. Suppose that $X$ is RC and $K_X+D$ is not pseudo-effective, then every component of $D$ with coefficient $1$ is a rational curve (in particular uniruled).
\end{lemma}
\begin{proof}
We run a $(K_X+D)$-minimal model program. Since $K_X+D$ is not pseudo-effective, by \cite[Theorem 1.1]{Fujino12} the minimal model program ends with a Mori fiber space which we denote by $g:X'\to Y$. Since $X'$ is an RC surface, $Y$ is either a point or a rational curve. If any component of $f_*^{-1}D$ is contracted during the minimal model program then by \cite[Theorem 2]{Kawamata91} that component must be a rational curve. We denote the strict transform of $f_*^{-1}D$ on $X'$ by $D'$ and denote by $D'_1,...,D'_m$ the irreducible components of $D'$ with coefficient $1$. 

If $Y$ is a point then $K_{X'}+D'$ is anti-ample. By the adjunction formula (cf. \cite[Proposition 3.9.2]{Corti07} we have 
$$(K_{X'}+D')|_{D'_i}=(K_{X'}+D'_i)|_{D'_i}+(D'-D'_i)|_{D'_i}=K_{D'_i}+{\rm Diff}_{D'_i}(D'-D'_i)$$
and ${\rm Diff}_{D'_i}(D'-D'_i)\ge 0$ by \cite[16.5]{Kollar92}. So $K_{D'_i}$ has negative degree, hence $D'_i$ is a rational curve. So for the rest of the proof we assume that $Y$ is a rational curve. For any $i$ if $D'$ does not dominate $Y$ then it is a component of a fiber of $g$, which is a rational curve by \cite[Lemma 3.7]{Debarre01}. If $D'_i$ dominates $Y$, then ${\rm deg}(g|_{D'_i})=>{\rm deg}(D'_i|_{F})>0$. Moreover we have
\begin{align*}
0> & {\rm deg}((K_{X'}+D')|_{F})>{\rm deg}((K_{X'}+D'_i)|_{F}) = {\rm deg}(K_{X'}|_{F})+{\rm deg}((D'_i)|_{F}) \\
= & {\rm deg}(K_{F})+{\rm deg}((D'_i)|_{F})=-2+{\rm deg}((D'_i)|_{F}),
\end{align*}
where the first inequality is by the fact that $-(K_{X'}+D')$ is $g$-ample, and the last equality is by the fact that $F\cong\mathbb{P}^1$. So we get that ${\rm deg}(g|_{D'_i})={\rm deg}((D'_i)|_{F})=1$, and since $Y$ is rational we know that $D'_i$ is rational. Hence every component of $D$ with coefficient $1$ is rational.
\end{proof}
\begin{proof}[Proof of Theorem \ref{hypersurfaceuniruledanydimension}]
By the argument before Lemma \ref{RCCsurface} it suffices to prove the theorem under the hypothesis that $(X,D)$ is dlt. We prove the theorem by induction on the dimension of $X$. When ${\rm dim}(X)=2$ this is proven in Lemma \ref{RCCsurface}. Suppose that the statement holds in any dimension  $k\in [2,n-1]$. Then in dimension $n$ we first run a minimal model program with scaling for  $(K_X+D)$. Since $K_X+D$ is not pseudo-effective there is an effective ample $\mathbb{Q}$-divisor $A$ such that 
\begin{itemize}
\item No component of $A$ is contained in ${\rm Supp}(D)$.
\item $K_X+D+A$ is still dlt and not pseudo-effective.
\item There exists a $\mathbb{Q}$-divisor $D_A$ such that $D+A\sim_{\mathbb{Q}}D_A$ and $(X,D_A)$ is klt. 
\end{itemize}
We run a $(K_X+D_A)$-MMP and by Theorem \ref{ExistMFS} it ends with a Mori fiber space as follows:
\begin{align}\label{MMP}
X=X_0 \overset{\text{$f_0$}}\dashrightarrow X_1\overset{\text{$f_1$}}\dashrightarrow ...\overset{\text{$f_{N-1}$}}\dashrightarrow X_N=X'\xrightarrow{g}Y.
\end{align}
Denote the strict transform of $D$, $A$ and $E_i$ on $X_k$ by $D^k$, $A^k$ and $E_i^k$ respectively. If for a certain $i$ and $k$, $E_i^k$ is contracted by $f_k$, then by \cite[Theorem 2]{Kawamata91} we know that $E_i^k$ is uniruled. By assumption on $(X,D+A)$ and \cite[Lemma 3.38]{KM98} we know that for any $k$, $(X_k,D^k+A^k)$, hence $(X_k, D^k)$, is dlt. Moreover it is easy to see that ${\rm rd}(X_i)\ge 2$ for any $i$. So we can assume that there is a morphism $f:X\to Y$ which is a Mori fiber space. By condition (3) in Definition \ref{DefMFS}, Lemma \ref{fiberklt} and \cite[Theorem 1]{Zhang06} we have that a general fiber of $f$ is RC (note that in this step we can actually work with $(X,
D_A)$ in this step, which is klt instead of dlt). Now we consider the following three cases respectively. \\

\noindent\emph{Case 1.} If ${\rm dim}(Y)=0$ then $-(K_X+D)$ is ample. So for any $E_i$ by the adjunction formula we have
$$(K_X+D)|_{E_i}=K_{E_i}+{\rm Diff}_{E_i}(D-E_i).$$
Hence $K_{E_i}+{\rm Diff}_{E_i}(D-E_i)$ is anti-ample and dlt, in particular $-K_{E_i}$ is big. Now if we do a $K_{E_i}$-minimal model program it would end with a Mori fiber space, in particular $E_i$ is uniruled.\\

\noindent\emph{Case 2.} If $1\le{\rm dim}(Y)\le n-2$, then for any $i$ we can assume that $E_i$ dominates $Y$. Indeed if this is not the case then for dimensional reasons $E_i$ is covered by fibers of $f$, and by the fact that the general fibers of $f$ are RC and \cite[Lemma 3.7]{Debarre01} we know that every fiber of $f$ is covered by rational curves. So we are done. Now for a general fiber $F$ of $f$ we have that $F$ is RC and $2\le{\rm dim}(F)\le n-1$. Suppose that $E_i|_F=\sum_lE_{F,i}^l$ where $E_{F,i}^l$ are the irreducible components of $E_i|_F$. By the adjunction formula we know that
$$(K_X+D)|_F=K_F+D|_F=K_F+E_{F,i}^l+(D|_F-E_{F,i}^l)$$ 
is anti-ample, so $-(K_F+E_{F,i}^l)$ is big for any $i$. After possibly doing a dlt modification for $(F,E_{F,i}^l)$ we can also assume that $(F,E_{F,i}^l)$ is dlt. By induction hypothesis we know that $E_{F,i}^l$ is uniruled for any $i$. Therefore $E_i$ is uniruled.\\

\noindent\emph{Case 3.} If ${\rm dim}(Y)=n-1$, then for the same reason as in \emph{Case 2} we can assume that $E_i$ dominates $Y$ for any $i$. After shrinking $X$ to its nonsingular locus, by generic smoothness the general fibers of $f$ are isomorphic to $\mathbb{P}^1$. Since ${\rm rd}(X)\ge 2$ we know that $Y$ is uniruled. So we only need to show that $f|_{E_i}$ has degree $1$. If ${\rm deg}(f|_{E_i})\ge 2$ then for a general fiber $F$ of $f$ we have 
$${\rm deg}((K_X+D)|_F)\ge{\rm deg}(K_F+E_i|_F)={\rm deg}(K_F)+{\rm deg}(E_i|_F)\ge -2+2=0,$$
in particular $-(K_X+D)$ cannot be $f$-ample. This is a contradiction, so we are done.
\end{proof}
\section{Rational connectedness of hypersurfaces}
Of course we can also ask whether certain positivity of $-(K_X+D)$ implies rational connectedness of components of $D$. This seems more complicated than uniruledness. We first point out that we cannot get RC-ness of components of $D$ by simply letting $X$ be RC in Theorem \ref{hypersurfaceuniruledanydimension}, even for log-smooth pairs in dimension $3$.
\begin{example}\label{P1bundleoverP2}
Let $g:X=\mathbb{P}(\mathcal{E})\to\mathbb{P}^2$ be the $\mathbb{P}^1$-bundle over $\mathbb{P}^2$ where $\mathcal{E}=\mathcal{O}_{\mathbb{P}^2}\oplus\mathcal{O}_{\mathbb{P}^2}(d)$ and $d\le -1$. Then $\omega_{X/\mathbb{P}^2}=g^*(\bigwedge^2\mathcal{E})\otimes\mathcal{O}_{\mathbb{P}(\mathcal{E})}(-2)$. Hence $K_X\sim_{\rm lin}(d-3)g^*H-2h$ where $H$ is a hyperplane in $\mathbb{P}^2$ and $h$ is the divisor class in $\mathbb{P}(\mathcal{E})$ induced by $\mathcal{O}_{\mathbb{P}(\mathcal{E})}(1)$. We take a general hypersurface $S\sim_{\rm lin}3H$ in $\mathbb{P}^2$ which is an elliptic curve. Let $D:=g^{-1}(S)$, then 
$$-K_X-D\sim_{\rm lin}-dg^*H+2h$$
which is big but obviously $g^{-1}(S)$ is not RC as $S$ is not rational.
\end{example}
However if we assume bigness of $\lfloor D\rfloor$ then we have the following result.
\begin{theorem}\label{Dbig}
Let $(X,D)$ be a plt pair. Suppose that $X$ is RC, $\lfloor D\rfloor$ is big and $K_X+D$ is not pseudo-effective. Then $\lfloor D\rfloor$ is RC.
\end{theorem}
\begin{lemma}\label{notcontractedtocurve}
Let $(X,D)$ be a $\mathbb{Q}$-factorial pair where $D$ is a big prime divisor. Let $\pi:X\to X'$ be a divisorial contraction such that $\rho(X/X')=1$. Then ${\rm dim}(\pi(D))={\rm dim}(D)$.
\end{lemma}
\begin{proof}
Suppose that $D$ is contracted to a lower-dimensional variety. If $D$ is $\pi$-nef then by the negativity lemma (cf. \cite[Lemma 3.39]{KM98}) we have $D=0$, which is a contradiction. If $D$ is not $\pi$-nef then there is a curve $\tilde{C}$, contracted by $\pi$, such that $\tilde{C}\cdot D<0$. We also observe that by bigness of $D$, $C'\cdot D\ge 0$ for a general curve $C'$ contracted by $\pi$. On the other hand whenever we choose a very ample divisor $H$ on $X$ we have that $C\cdot H>0$ for any curve $C$ in $X$. This is a contradiction to the assumption $\rho(X/X')=1$. So we are done.
\end{proof}
\begin{proof}[Proof of Theorem \ref{Dbig}]
We let $S:=\lfloor D\rfloor$ and $B:=\{D\}$. We do the same minimal model program as in the proof of Theorem \ref{hypersurfaceuniruledanydimension} as follows
\begin{align}\label{MMP2}
X=X_0 \overset{\text{$f_0$}}\dashrightarrow X_1\overset{\text{$f_1$}}\dashrightarrow ...\overset{\text{$f_{N-1}$}}\dashrightarrow X_N=X'\xrightarrow{g}Y,
\end{align}
and denote the strict transform of $D$, $S$ and $B$ on $X_i$ by $D_i$, $S_i$ and $B_i$ respectively. Certainly $S_i$ is big, hence it cannot be contracted. Moreover by the adjunction formula we have 
$$(K_{X_i}+D_i)|_{S_i}=K_{S_i}+{\rm Diff}_{S_i}(B_i),$$
so $(S_i,{\rm Diff}_{S_i}(B_i))$ is klt for any $i$. Therefore it could be assumed the existence of a morphism $f: X\to Y$ which is a Mori fibre space. Now since $X$ is RC then so is $Y$, and since $S$ is big it must dominate $Y$. Next we consider the following three cases. \\

\noindent\emph{Case 1.} If ${\rm dim}(Y)=0$ then $-(K_X+D)$ is ample, so by the adjunction formula we have that $-(K_S+{\rm Diff}_{S}(B))$ is ample and $(S,{\rm Diff}_{S}(B))$ is klt. Then by \cite[Theorem 1]{Zhang06} $S$ is RC. \\

\noindent\emph{Case 2.} If $1\le {\rm dim}(Y)\le n-2$ we denote a general fiber of $f$ by $F$. Then $(K_X+D)|_F=K_F+D_F$ is anti-ample. By Koll\'{a}r-Shokurov connectedness lemma (cf. \cite[Theorem 2.3.1]{Prokhorov99}) we see that $S|_F$ is connected. Now we do a Stein factorization of $f|_S$ and denote it as
$$S\xrightarrow{g}Z\xrightarrow{h}Y.$$
Since $S|_F$ is connected we know that $h$ is birational. So $Z$ is RC as RC-ness is a birational invariant (cf. \cite[Chapter IV, Proposition 3.3]{Kollar96}).

On the other hand since $-(K_X+D)$ is $f$-ample, $K_{S}+{\rm Diff}_{S}(B)$ is $f|_S$-ample, hence $g$-ample. So if we denote the fiber of $f|_S$ over a general point $z$ of $Z$ by $S_z$, by Lemma \ref{fiberklt} we know that $K_{S_z}+{\rm Diff}_{S}(B)|_{S_z}$ is klt and anti-ample. Hence $S_z$ is RC. 

Finally by \cite[Corollary 1.3]{GHS03} we know that $S$ is RC.\\

\noindent\emph{Case 3.} If ${\rm dim}(Y)=n-1$ then by the same argument as in the proof of Theorem \ref{hypersurfaceuniruledanydimension} we have that ${\rm deg}(f|_S)=1$. Moreover since $X$ is RC we know $Y$ is RC, hence $S$ is RC.
\end{proof}
Going back to Example \ref{ruledsurface}, we see that $-(K_X+D)$ being big and semiample does not imply RC-ness of components of $D$. Nevertheless we can ask what happens if we assume RC-ness of $X$ in addition. Clearly we cannot expect that every component of $D$ with coefficient $1$ is RC. For example if we take $D=g^{-1}(S)+h$ in Example \ref{P1bundleoverP2}, then $-K_X-D=-dg^*H+h$ which is big and semiample but $g^{-1}(S)$ is still not RC. However on the other hand if $(X,D)$ is dlt then by Koll\'{a}r-Shokurov connectedness lemma the union of all the components of $D$ with coefficient $1$ is connected. So we can still ask whether such locus is rationally chain connected.
\begin{question}\label{RCbigsemiample}
Let $(X,D)$ be a dlt pair where $D=\sum_iE_i+\sum_ja_jF_j$ such that $E_i$ and $F_j$ are prime divisors and $a_j\in(0,1)$. Suppose that $X$ is RC and $-(K_X+D)$ is big and semiample, then is $\bigcup_iE_i$ RCC?
\end{question}
Unfortunately we do no have an answer to \emph{Question} \ref{RCbigsemiample} in general so far. Nevertheless we are able to show that the answer is positive for certain cases of threefolds and toric varieties. 
\begin{theorem}\label{P3}
Let $(X,D)$ be 3-dimensional plt pair where $D$ a prime divisor on $X$.  Suppose that $X$ is RC and $-(K_X+D)$ is Cartier, nef and big. Then $D$ is RC.
\end{theorem}
\begin{lemma} \label{ruledtorational}
Let $S$ be a normal surface with rational singularities. If $S$ is birational to a ruled surface and $H^1(S,\mathcal{O}_S)=0$ then $S$ is a rational surface.
\end{lemma}
\begin{proof}
We do a resolution $f:S'\to S$ for $S$. Since $S$ is birational to a ruled surface, so is $S'$. In particular $H^0(S',\mathcal{O}_{S'}(2K_{S'}))=0$. On the other hand we have 
\begin{align*}
h^1(S',\mathcal{O}_{S'})=h^1(S',f^*\mathcal{O}_{S})=h^1(S,f_*f^*\mathcal{O}_{S})=h^1(S,\mathcal{O}_S)=0,
\end{align*}
where the second equality is by the assumption that $S$ has rational singularities. So by a theorem of Castelnuovo (cf. \cite[Theorem V.1]{Beauville96}) we know that $S$ is a rational surface.
\end{proof}
\begin{lemma}\label{irregularity0}
Let $(X,D)$ be a plt pair where $X$ has dimension $n\ge 2$ and $D$ prime divisor on $X$. Suppose that $X$ is RC and $-(K_X+D)$ is Cartier, nef and big. Then $H^1(D,\mathcal{O}_D)=0$.
\end{lemma}
\begin{proof}
We have the following short exact sequence
$$0\to\mathcal{O}_X(K_X)\to\mathcal{O}_X(K_X+D)\to\mathcal{O}_D(K_D)\to 0$$
which yields the following long exact sequence
$$...\to H^{n-2}(X,\mathcal{O}_X(K_X+D))\to H^{n-2}(D,\mathcal{O}_D(K_D))\to H^{n-1}(X,\mathcal{O}_X(K_X))\to...$$
Since $X$ is klt and RC we know $H^{n-1}(X,\mathcal{O}_X(K_X))=H^1(X,\mathcal{O}_X)=0$. By Kawamata-Viehweg vanishing we also have 
$$H^{n-2}(X,\mathcal{O}_X(K_X+D))=H^2(X,\mathcal{O}_X(-D))=H^2(X,\mathcal{O}_X(K_X+(-K_X-D)))=0$$
as $-K_X-D$ is nef and big by assumption. So we get
$$H^{n-2}(D,\mathcal{O}_D(K_D))=H^1(D,\mathcal{O}_D)=0.$$
\end{proof}
\begin{proof}[Proof of Proposition \ref{P3}]
By Lemma \ref{irregularity0} we have $H^1(D,\mathcal{O}_D)=0$. On the other hand by Theorem \ref{hypersurfaceuniruledanydimension} we know that $D$ is birational to a ruled surface. So by Lemma \ref{ruledtorational} we are done.
\end{proof}
Before showing the result for toric varieties we present the following proposition, which we hope to be of independent interest.
\begin{proposition}\label{RCChypersurface}\label{P2}
Suppose that we have a pair $(X,D)$ where $D=\sum_iE_i+\sum_ja_jF_j$ such that $E_i$ and $F_j$ are prime divisors and $a_j\in(0,1)$. Suppose that $(X,D)$ is dlt, $-(K_X+D)$ is big and semiample, and there is no lc center (or equivalently, non-klt center) of $(X,D)$ that is contained in $\mathbf{B}_+(-(K_X+D))$. Then $E_i$ is RC for any $i$.
\end{proposition}
To prove this we need the following lemma which is a slight modification of \cite[Theorem 1]{Zhang06}. 
\begin{lemma}\label{dltRC}
Let $(X,D)$ be a dlt pair and suppose that $-(K_X+D)$ is ample. Then $X$ is RC. 
\end{lemma}
\begin{proof}
By \cite[Proposition 2.43]{KM98} we can perturb $D$ such that $(X,D)$ is klt and $-(K_X+D)$ still stays ample. So by \cite[Theorem 1]{Zhang06} we are done.
\end{proof}
\begin{proof}[Proof of Proposition \ref{RCChypersurface}]
By assumption there exists an effective $\mathbb{Q}$-divisor $H$ such that $H\sim_{\mathbb{Q}}-(K_X+D)$ and $H\sim_{\mathbb{Q}}A+G$, where $A$ is ample and $G$ is effective. We have
$$0\sim_{\mathbb{Q}}K_X+D+H\sim_{\mathbb{Q}}K_X+D+(1-\epsilon)H+\epsilon(A+G).$$ 
Moreover we can arrange $\epsilon$, $H$ and $G$ such that 
\begin{itemize}
\item $E_j\not\subseteq {\rm Supp}(G)$ for any $j$.
\item $(X, D+(1-\epsilon)H+\epsilon G)$ is dlt and the only components of $D+(1-\epsilon)H+\epsilon G$ with coefficient 1 are the $E_j$.
\end{itemize}
Now by adjunction there exists an effective $\mathbb{Q}$-divisor $D_{E_i}$ such that 
$$(K_X+D+(1-\epsilon)H+\epsilon G)|_{E_i}\sim_{\mathbb{Q}}K_{E_i}+D_{E_i}$$
and $(K_{E_i},D_{E_i})$ is dlt. By construction we have $K_X+D+(1-\epsilon)H+\epsilon G\sim_{\rm lin}-\epsilon A$ is anti-ample, so $K_{E_i}+D_{E_i}$ is anti-ample as well. Then by Lemma \ref{dltRC}, $E_i$ is RC.
\end{proof}
\begin{theorem}\label{P4}
Let $(X,D)$ be a plt pair where $X$ is a toric variety. Suppose that $-(K_X+D)$ is big and semiample. Then $\lfloor D\rfloor$ is RC.
\end{theorem}
\begin{proof}
If $\lfloor D\rfloor$ is toric invariant then we are done. If not then by \cite[Lemma 15.1.8]{CLS11} we know that $\lfloor D\rfloor$ is $\mathbb{Q}$-linearly equivalent to a linear combination of Cartier toric invariant divisors with nonnegative coefficients, in particular $\lfloor D\rfloor\not\subseteq\mathbf{B}_+(-(K_X+D))$. So by Proposition \ref{P2} we are done.
\end{proof}
%\begin{remark}
%Although no counter-example is found so far, the author believes that the answer to Question \ref{RCbigsemiample} is negative in general.
%\end{remark}
\bibliographystyle{alpha}
\bibliography{P}  
\end{document}